\documentclass[letterpaper,11pt]{article}
\usepackage{color}

\usepackage[small]{titlesec}
\usepackage{amsthm,amsmath,amssymb,amscd,bbm}
\usepackage[all,cmtip]{xy}
\usepackage{booktabs}
\usepackage[hyperfootnotes=false, colorlinks, linkcolor={blue}, citecolor={magenta}, filecolor={blue}, urlcolor={blue}]{hyperref}
\usepackage[overload]{textcase} 
\usepackage{comment}
\usepackage{tikz-cd}

\allowdisplaybreaks
\newcommand{\A}{{\mathbb A}}
\newcommand{\Q}{{\mathbb Q}}
\newcommand{\Z}{{\mathbb Z}}

\newcommand{\C}{{\mathbb C}}

\newcommand{\ee}{\mathrm{e}}
\newcommand{\ii}{\mathrm{i}}
\newcommand{\p}{\mathfrak p}

\newcommand{\Sp}{{\rm Sp}}

\newcommand{\trace}{{\rm tr}}

\newcommand{\HH}{\mathbb{H}}

\newcommand{\mat}[4]{\begin{bsmallmatrix}#1&#2\\#3&#4\end{bsmallmatrix}}

\newcommand{\forget}[1]{}
\def\qdots{\mathinner{\mkern1mu\raise0pt\vbox{\kern7pt\hbox{.}}\mkern2mu
\raise3.4pt\hbox{.}\mkern2mu\raise7pt\hbox{.}\mkern1mu}}

\newenvironment{bsmallmatrix}{\left[\begin{smallmatrix}}{\end{smallmatrix}\right]}

\let\originalleft\left
\let\originalright\right
\renewcommand{\left}{\mathopen{}\mathclose\bgroup\originalleft}
\renewcommand{\right}{\aftergroup\egroup\originalright}

\newtheorem{lemma}{Lemma}[section]
\newtheorem{theorem}[lemma]{Theorem}

\newtheorem{proposition}[lemma]{Proposition}

\usepackage{tikz}
\usetikzlibrary{shapes.geometric, arrows}

\tikzstyle{globaldark} = [rectangle, rounded corners,
minimum width=3cm, 
minimum height=1cm, 
text centered, 
text width=3cm, 
draw=black, 
fill=cyan!60]

\tikzstyle{ichar} = [rectangle, rounded corners,
minimum width=2.5cm, 
minimum height=1cm, 
text centered, 
text width=2.5cm, 
draw=black, 
fill=cyan!60]

\tikzstyle{globallight} = [rectangle, rounded corners,
minimum width=3cm, 
minimum height=1cm, 
text centered, 
text width=3cm, 
draw=black, 
fill=cyan!30]

\tikzstyle{localdark} = [rectangle, rounded corners,
minimum width=3cm, 
minimum height=1cm, 
text centered, 
text width=3cm, 
draw=black, 
fill=purple!45]

\tikzstyle{locallightfit} = [rectangle, rounded corners,
minimum width=3cm, 
minimum height=1.5cm, 
text centered, 
text width=3cm, 
draw=black, 
fill=purple!30]

\tikzstyle{locallight} = [rectangle, rounded corners,
minimum width=3cm, 
minimum height=1cm, 
text centered, 
text width=3cm, 
draw=black, 
fill=purple!30]

\tikzstyle{classical} = [rectangle, rounded corners,
minimum width=3cm, 
minimum height=1cm, 
text centered, 
text width=3cm, 
draw=black, 
fill=green!30]

\tikzstyle{dchar} = [rectangle, rounded corners,
minimum width=2.5cm, 
minimum height=1cm, 
text centered, 
text width=2.5cm, 
draw=black, 
fill=green!30]

\tikzstyle{arrow} = [thick,->,>=stealth]

\begin{document}

\thispagestyle{empty}
\begin{center}
 {\bf\Large On the rationality of a paramodular Siegel Eisenstein series}

 \vspace{3ex}
 Erin Pierce

\vspace{3ex}
\begin{minipage}{80ex}
{\small ABSTRACT:  We consider the rationality of the Fourier coefficients of a particular paramodular Siegel Eisenstein series of level $N^2$ with weight $k\geq 4$.  We show that the coefficients lie in a number field.}
\end{minipage}

\end{center}

\tableofcontents

\section{Introduction}

Let $k\geq 4$ be an even integer. Siegel \cite{Siegel1935} considered the weight $k$ Eisenstein series
\begin{equation}\label{introeq1}
 E_k(Z)=\sum_{\gamma=\begin{bsmallmatrix}
     A&B\\C&D
 \end{bsmallmatrix}\in P(\Z)\backslash\Sp(4,\Z)}\det(CZ+D)^{-k},
\end{equation}
where $Z$ is an element of the Siegel upper half space $\HH_2$ of degree $2$ and $P$ is the Siegel parabolic subgroup of $\Sp(4)$ consisting of matrices whose lower left block is zero.  Starting from a primitive Dirichlet character $\eta$ of conductor $N$, in  \cite{pierceschmidt2025siegel} we defined a generalization $E_{k,\eta}$, which is a degree~$2$, weight-$k$ Siegel modular form with respect to $K(N^2)$, where
\begin{align}\label{globaldefofparamodulargroup}
 K(M)=\Sp(4,\Q)\cap\begin{bsmallmatrix}\Z&M\Z&\Z&\Z\\\Z&\Z&\Z&M^{-1}\Z\\\Z&M\Z&\Z&\Z\\M\Z&M\Z&M\Z&\Z\end{bsmallmatrix}.
\end{align} is the paramodular group of level $M$. More precisely, 
\begin{equation}\label{introeq5}
 E_{k,\eta}(Z)=\frac12\sum_{b\in(\Z/N\Z)^\times}\eta(b)E_k(Z,K(N^2),\mathcal{C}_0(bN)),\qquad\text{where }\mathcal{C}_0(x)=\begin{bsmallmatrix}1\\&1\\&x&1\\x&&&1\end{bsmallmatrix}
\end{equation} 
and
\begin{equation}\label{rewritesummationeq11}
 E_k(Z,\Gamma,h):=\sum_{\gamma\in (\Gamma\cap h^{-1}P^{\scriptscriptstyle{1}}(\Q)h)\backslash \Gamma}\det(J(h\gamma,Z))^{-k}.
\end{equation}  The work \cite{pierceschmidt2025siegel} explicitly calculates the Fourier coefficients of this Eisenstein series; see Theorem~\ref{Fourierexpansiontheorem} below. 

Siegel \cite{Siegel1939} gave a formula for the Fourier coefficients of $E_k(Z)$, which is sufficiently explicit to see that they are rational numbers. It follows from general results of Shimura (see Lemma 10.5 of \cite{shimura2000arithmeticity} and also Lemma 6.10 of \cite{schmidtsahapitale2021}) that the Fourier coefficients of a suitably normalized  weight-$k$ modular form with respect to a principal congruence subgroup $\Gamma(N)$ are contained in some number field.  In this work, using the explicit formula in \cite{pierceschmidt2025siegel}, we give the following refinement of these general results.

\begin{theorem}\label{maintheorem}
    Let $E_{k,\eta}$ be as in \eqref{introeq5}, where $\eta$ is a primitive Dirichlet character of conductor $N$.  
    \begin{enumerate}
        \item The Fourier coefficients of $E_{k,\eta}$ lie in the number field generated by $\ii$, the values of $\eta$, and the $N$-th roots of unity, i.e. in $\Q(\ii, \eta,\zeta_N)$.
        \item For fixed $T=\begin{bsmallmatrix}
            n&r/2\\r/2&m
        \end{bsmallmatrix}$, write $r^2-4nm=Df^2$ with $D$ a fundamental discriminant and $f$ an integer.  Let $\chi_D$ be the quadratic Dirichlet character mod $|D|$, $\alpha$ be the unique primitive Dirichlet character associated with $\chi_D\eta$, and $\beta$ be the unique primitive Dirichlet character associated with $\eta^2$. Consider the Fourier coefficient $a(T)$ from Theorem \ref{Fourierexpansiontheorem}. 
        \begin{enumerate}
            \item If ${\rm rank}(T)=1$, then \begin{equation}\label{rank1rationality}
                a(T)\in \Q(\eta, G(\eta)),
            \end{equation} the field generated by the values of $\eta$ and the Gauss sum of $\eta$.
            \item   If ${\rm rank}(T)=2$, then \begin{equation}\label{rank2rationality}
                a(T)\in \Q(\eta,\sqrt{|D|}G(\alpha),G(\beta),\ii),
            \end{equation} the field generated by the values of $\eta$, $\sqrt{|D|}$ times the Gauss sum of $\alpha$, the Gauss sum of $\beta$, and $\ii$.
        \end{enumerate}
    \end{enumerate}
\end{theorem}  
\noindent Note that $L:=\Q(\eta,\sqrt{|D|}G(\alpha),G(\beta),\ii)$ is a subfield of $\Q(\eta,\zeta_N,\ii)$.  In general, it is a proper subfield.  For example, if $\eta^2=1$ and $\alpha=\chi_D\eta$ is primitive, then $L=\Q(\sqrt{N},\ii)\subseteq\Q(\zeta_N,\ii)$. \\

\noindent \textbf{Acknowledgements.} I would like to thank Ralf Schmidt for many helpful discussions and for helping to improve the exposition.

\section{Notation}

We will use the notation
\begin{equation}\label{deltaeq}
 \delta_*=\begin{cases}
           1&\text{if the condition $*$ is satisfied},\\
           0&\text{if the condition $*$ is not satisfied}.
          \end{cases}
\end{equation}  We will use the notation $\zeta_j$ to denote a primitive $j$-th root of unity.  For a Dirichlet character $\eta$, we use $\Q(\eta)$ to denote the extension of $\Q$ generated by the values of $\eta$.  We let $v$ denote the $p$-adic valuation.

For a nonzero integer $r=\pm \prod p^{\alpha_p}$ and a positive integer $N$, we set
\begin{equation}\label{Nnotation}
    r_{\scriptscriptstyle N}=\prod_{p\mid N} p^{\alpha_p} \quad \text{ and }\quad r_{\scriptscriptstyle {\hat N}}=\pm\prod_{p\nmid N} p^{\alpha_p}.
\end{equation} Clearly, $r=r_{\scriptscriptstyle N} r_{\scriptscriptstyle {\hat N}}$.  For a Dirichlet character $\eta$ and an integer $\ell$, we let
\begin{equation}\label{divisorsumwithcharacter}
    \sigma_{\ell,\eta}(a)=\sum_{d\mid a}\eta^{-1}(d)\, d^\ell,
\end{equation} for any nonzero integer $a$. 

\section{Gauss sums and epsilon factors}
Recall that the Gauss sum of a Dirichlet character $\chi\bmod m$ is defined as
\begin{equation}\label{Gausssumeq1}
 G(\chi)=\sum_{k=0}^{m-1}\chi(k)\ee^{2\pi \ii k/m}=\sum_{k\in(\Z/m\Z)^\times}\chi(k)\ee^{2\pi \ii k/m}.
\end{equation}
For the principal character we have $G(\chi)=1$.  By definition, the Gauss sum of the Legendre symbol mod $p$ takes values in $\Q(\zeta_p)$.  By a famous result of Gauss,
\begin{equation}
    G\left(\left(\frac{\cdot}{p}\right)\right)=\begin{cases}
        \sqrt{p} & \text{if } p\equiv 1\bmod 4\\
        \ii\sqrt{p} & \text{if } p\equiv 3\bmod 4.
    \end{cases}
\end{equation}  Hence, for odd primes $p$,
\begin{equation}\label{sqrtoddp}
    \sqrt{p}\in\Q(\zeta_p,\ii)\overset{p\mid N}{\subseteq}\Q(\zeta_N,\ii).
\end{equation}  We also recall that 
\begin{equation}\label{sqrtevenp}
    \sqrt{2}\in\Q(\zeta_8)\overset{v_2(N)\geq 3}{\subseteq}\Q(\zeta_N).
\end{equation}  The following properties are well known from \cite{TateThesis} (see also Sect.~3.1 of \cite{Bump1997} for more contemporary notation).  Let $\psi$ be an additive character of $\Q_p$ with conductor $\Z_p$ and let $\chi$ be a character of $\Q_p^\times$ with conductor exponent $a(\chi)$.
\begin{align}
    &\text{If }a(\chi)>0\text{, then }\nonumber\\&\qquad\varepsilon(\tfrac{1}{2},\chi,\psi)=p^{a(\chi)/2}\chi(x)^{-1}\int\limits_{\Z_p^\times} \chi(\mu^{-1})\psi(x\mu)\, d\mu\text{ with }x\text{ such that }-v(x)=a(\chi). \label{epsilonfactordef}\\
    &\text{If }\chi\text{ is unramified, then }  \varepsilon(\tfrac{1}{2},\chi,\psi)=1.\\
    &\text{If }\chi \text{ is quadratic, we have }\varepsilon(\tfrac{1}{2},\chi,\psi)^2=\chi(-1)=\pm 1. \label{rationalityeq17}\\
    &\text{For } \mu \text{ an unramified character}, \varepsilon(s,\mu\chi,\psi)=\mu(p)^{a(\chi)}\varepsilon(s,\chi,\psi). \label{rationalityeq18}
\end{align}  Using the integral formula \eqref{epsilonfactordef} and an inductive argument, one can easily prove the following relationship between the Gauss sum $G(\eta)$ and local $\varepsilon$-factors.

\begin{proposition}\label{epsilonfactorgausssumprop}
    Let $N=\prod p^{n_p}$.  Let $\eta$ be a primitive Dirichlet character mod $N$ corresponding to the adelic character $\chi=\otimes \chi_p$.  Let $\psi=\prod \psi_p$ be the character of $\Q/\A$ for which $\psi_\infty=e^{-2\pi\ii x}$.  Then
    \begin{equation}\label{epsilonprop}
        \prod\limits_{p<\infty}\varepsilon(\tfrac{1}{2},\chi_p,\psi_p)=\frac{\eta(-1)G(\eta)}{\sqrt{N}}.
    \end{equation}
\end{proposition}

\section{Special values of Dirichlet \texorpdfstring{$L\,$-functions}{}}
It is well known (see Corollary VII.1.10 of \cite{Neukirch1999}) that
\begin{equation}\label{specialvalueseq2}
 \zeta(2k)=(-1)^{k-1}\frac{(2\pi)^{2k}}{2(2k)!}B_{2k}
\end{equation}
for any positive integer~$k$, where $B_{2k}$ represents a Bernoulli number. For a non-trivial, primitive Dirichlet character $\eta\bmod m$, we define the generalized Bernoulli numbers $B_{k,\eta}$ by
\begin{equation}\label{specialvalueseq3}
 \sum_{a=1}^m\eta(a)\frac{t\ee^{at}}{\ee^{mt}-1}=\sum_{k=0}^\infty B_{k,\eta}\frac{t^k}{k!}.
\end{equation}
Let
\begin{equation}\label{specialvalueseq4}
 \epsilon=\begin{cases}
           0&\text{if }\eta(-1)=1,\\
           1&\text{if }\eta(-1)=-1.
          \end{cases}
\end{equation}
Then $B_{k,\eta}=0$ if $k\not\equiv\epsilon\bmod2$.  By Corollary~VII.2.10 of \cite{Neukirch1999}, for a positive integer $k$ we have
\begin{equation}\label{specialvalueseq6}
 L(k,\eta)=(-1)^{1+\frac{k-\epsilon}2}\frac{G(\eta)}{2\ii^\epsilon}\left(\frac{2\pi}m\right)^k\frac{B_{k,\bar\eta}}{k!}\qquad\text{if }k\equiv\epsilon\bmod2.
\end{equation}

\section{Rationality of \texorpdfstring{$K(s,T,\chi_p)$}{}}
Let $\chi=\otimes\chi_p$ be the idele class character corresponding to the primitive Dirichlet character $\eta$ of conductor $N=\prod p^{n_p}$ so that $a(\chi_p)=n_p$.  Let $T=\begin{bsmallmatrix}
    n&r/2\\r/2&m
\end{bsmallmatrix}$ be a non-zero positive semi-definite matrix with integers $n,r,m$ such that $N^2\mid m$.  For the rest of this section, we fix a prime $p$.  For $s\in\C$, we define
\begin{equation}\label{defofK}
 K(s,T,\chi_p):=\sum_{j=1-n_p}^\infty\;p^{j(2-s)}\int\limits_{S(j+1,n_p)}\chi_p\left(n\mu^{-1}+rp^{-n_p}+m\mu p^{-2n_p}\right)\,d\mu,
\end{equation} where 
\begin{equation}\label{defofS}
    S(j+1,n_p)=\{\mu\in\Z_p^\times \mid v(n+r\mu p^{-n_p}+m\mu^2 p^{-2n_p})=j\}.
\end{equation} In this section, we will state the following without full proof.  For details, see \cite{Pierce2026}.

\begin{proposition}\label{rationalityofKprop}
    Define $K(s,T,\chi_p)$ as in \eqref{defofK}.  Then
     \begin{equation}\label{rationalityofKeq1}
         K(k,T,\chi_p)\in \Q(\eta).
     \end{equation}
 \end{proposition}

\begin{proof}
    Note that we set $s=k$ so that the value $p^{2-s}$ is rational.  To prove \eqref{rationalityofKeq1}, it suffices to show that the sum in \eqref{defofK} is finite, which is true for one of two reasons. In some cases, the set $S(j+1,n_p)$ is empty for all but one value of $j$.  In all other cases, the integral equals zero for large $j$ since $\chi_p$ is ramified.  We will give illustrative examples for both situations.

As an example for the first situation, suppose $v(n)<\min(v(r)-n_p,v(m)-2n_p)$.  Then 
\begin{equation}
    S(j+1,n_p)=\{\mu\in\Z_p^\times\mid v(n)=j\}=\Z_p^\times.
\end{equation} Hence, the sum in $K(s,T,\chi_p)$ collapses to the term $j=v(n)$.

As an example for the second situation, suppose $v(n)=v(r)-n_p<v(m)-2n_p$.  Recall that $r^2-4mn=Df^2$ for $D$ a fundamental discriminant and $f$ an integer. We have the identity
\begin{align}
    n\mu^{-1}+rp^{-n_p}+m\mu p^{-2n_p}&=n\mu\left(\left(\mu^{-1}+\frac{r}{2np^{n_p}}\right)^2-\frac{z^2f^2}{4n^2 p^{2n_p}}\right). \label{calcofKeq13}
\end{align}  Assume $j>v(n)$.  Assume first that \underline{$p$ is odd}.  The following identity
\begin{equation}\label{calcofKeq8}
    p\Z_p\ni \frac{m}{np^{2n_p}}=\left(\frac{-r}{2np^{n_p}}\right)^2-\frac{Df^2}{4n^2p^{2n_p}},
\end{equation} implies $D=z^2$ for some $z\in\Z_p^{\times}$. Now
\begin{equation}\label{calcofKeq10}
    \frac{4mn}{r^2}\in p\Z_p\implies z^2\frac{f^2}{r^2}\in 1+p\Z_p\implies \left(z\frac{f}{r}-1\right)\left(z\frac{f}{r}+1\right)\in p\Z_p.
\end{equation}  Replacing $z\mapsto -z$ if necessary, we can assume $v(zf-r)=v(r)$ and $v(zf+r)>v(r)$, and hence
\begin{equation}\label{calcofKeq12}
    \frac{zf-r}{2np^{n_p}}\in\Z_p^\times \quad \text{ and } \quad \frac{-zf-r}{2np^{n_p}}\in p\Z_p.
\end{equation}  By \eqref{calcofKeq13} and \eqref{calcofKeq12} the integral in \eqref{defofK} equals
\begin{align}\label{calcofKeq17}
&\int\limits_{\frac{2np^{n_p}}{zf-r}\left(1+p^{j-v(n)}\Z_p^\times\right)}\chi_p(n\mu)\chi_p\left(\left(\mu^{-1}+\frac{r}{2np^{n_p}}\right)^2-\frac{z^2f^2}{4n^2 p^{2n_p}}\right)\, d\mu\nonumber\\
    &\qquad\qquad=\int\limits_{1+p^{j-v(n)}\Z_p^\times}\chi_p\left(\mu\frac{r^2}{2(zf-r)p^{n_p}}\right)\,\chi_p\left(\left(\frac{zf-r}{r}\mu^{-1}+1\right)^2-\frac{z^2f^2}{r^2}\right)\, d\mu\nonumber\\ 
    & \qquad\qquad\overset{j\gg 0}{=}\chi_p\left(\frac{r^2}{2(zf-r)p^{n_p}}\right)\, p^{v(n)-j}\, \int\limits_{\Z_p^\times}\,\chi_p\left(\left(\frac{zf-r}{r}(\mu p^{j-v(n)}+1)+1\right)^2-\frac{z^2f^2}{r^2}\right)\, d\mu\nonumber\\
    &\qquad\qquad \overset{j\gg 0}{=}\chi_p\left(p^{j-v(n)-n_p}\right)\, p^{v(n)-j}\, \int\limits_{\Z_p^\times}\,\chi_p\left(zf\mu\right)\, d\mu=0.
\end{align}   Hence, the sum in \eqref{defofK} is indeed finite.
%
%
%

Assume now that \underline{$p=2$}.  The identity
\begin{equation}\label{Krationalityp=2eq4}
    8\Z_2\ni \frac{m}{np^{2n_p-2}}=\left(-\frac{r}{np^{n_p}}\right)^2 - \frac{Df^2}{n^2p^{2n_p}},
\end{equation} implies $v(D)$ is even, hence $v(D)=0$ or $2$.  If \underline{$v(D)=0$}, $D=z^2$ for some $z\in\Z_2^\times$ and $v(f)=v(r)$. Notice that if
\begin{equation}\label{Krationalityp=2eq9}
    \mu^{-1}+\frac{r-zf}{np^{n_p+1}}, \quad \mu^{-1}+\frac{r+zf}{np^{n_p+1}},
\end{equation} are both in $\Z_2^\times$ or both in $2\Z_2$, their difference $\frac{-zf}{np^{n_p}}$ must be in $2\Z_2$, which is a contradiction.  Replacing $z\mapsto -z$ if necessary, we can assume 
\begin{equation}\label{Krationalityp=2eq11}
    \frac{r-zf}{np^{n_p+1}}\in\Z_2^\times.
\end{equation} By \eqref{calcofKeq13} and \eqref{Krationalityp=2eq11}, the integral in \eqref{defofK} equals
\begin{align}\label{Krationalityp=2eq13}
    &\int\limits_{\frac{np^{n_p+1}}{zf-r}\left(1+2^{j-v(n)}\Z_2^\times\right)}\chi(n\mu)\chi\left(\left(\mu^{-1}+\frac{r}{np^{n_p+1}}\right)^2-\frac{z^2f^2}{n^2p^{2n_p+2}}\right)\, d\mu\nonumber\\
    &\qquad\qquad \overset{j\gg 0}{=}\int\limits_{1+2^{j-v(n)}\Z_2^\times}\chi\left(\frac{1}{(zf-r)p^{n_p+1}}\right)\chi\left(\left((zf-r)\mu^{-1}+r\right)^2-z^2f^2\right)\, d\mu\nonumber\\
    &\qquad\qquad =p^{v(n)-j}\int\limits_{\Z_2^\times}\chi\left(\frac{1}{p^{n_p+1}}\right)\chi\left(2^{j-v(n)}\mu\left(2zf\left(1+\frac{(zf-r)\mu2^{j-v(n)}}{2zf}\right)\right)\right)\, d\mu\nonumber\\
    &\qquad\qquad \overset{j\gg 0}{=}p^{v(n)-j}\int\limits_{\Z_2^\times}\chi\left(\frac{2^{j-v(n)+1} zf}{p^{n_p+1}}\right)\chi\left(\mu \right)\, d\mu=0.
\end{align}  If \underline{$v(D)=2$}, $D=4z^2$ for some $z\in \Z_2^\times$. Then by a similar calculation, for $j\gg 0$, by \eqref{calcofKeq13}, the integral in \eqref{defofK} equals
\begin{align}\label{Krationalityp=2eq23}
    &\int\limits_{\frac{np^{n_p+1}}{2zf-r}\left(1+2^{j-v(n)}\Z_2^\times\right)}\chi(n\mu)\chi\left(\left(\mu^{-1}+\frac{r}{np^{n_p+1}}\right)^2-\frac{z^2f^2}{n^2p^{2n_p}}\right)\, d\mu=0.
\end{align}  Hence, in all cases, the sum in \eqref{defofK} is finite.
\end{proof} 
%
%
%


%

\section{Rationality of the Fourier coefficients}
For $D$ a fundamental discriminant, let $\chi_D$ be the primitive quadratic Dirichlet character mod~$|D|$.  Recall that $\eta$ is a primitive Dirichlet character of conductor $N$.  For integers $e,f$ with $e\mid f$, we define
\begin{equation}\label{tildeHetadef}
 \tilde H_{D,s,\eta}(e,f)=\sum_{d\mid e}\eta(d)^{-1}d^{s-1}\sum_{g\mid\frac fd}\mu(g)\chi_D(g)\eta(g)^{-1}g^{s-2}\sum_{h\mid\frac{f}{dg}}\eta(h)^{-2}h^{2s-3},
\end{equation} where $\mu$ is the M\"obius $\mu$ function.    By $\chi_D\eta$ we mean the character of conductor $\ell:={\rm lcm}(|D|,N)$ such that $(\chi_D\eta)(x)=\chi_D(x)\eta(x)$ for $x$ relatively prime to $\ell$. This character is not necessarily primitive.  The same can be said about $\eta^2$, which has conductor $N$.  We let $\alpha$ be the unique primitive Dirichlet character mod $a$ and let $\beta$ be the unique primitive Dirichlet character mod $b$ such that the following diagrams commute.
 \begin{equation}\label{rationalityeq10}
     \begin{tikzcd}
    (\Z/\ell\Z)^\times \arrow{r}{} \arrow[swap]{dr}{\chi_D\eta} & (\Z/a\Z)^\times \arrow{d}{\alpha} \\
     & \C^\times
  \end{tikzcd} \qquad
  \begin{tikzcd}
    (\Z/N\Z)^\times \arrow{r}{} \arrow[swap]{dr}{\eta^2} & (\Z/b\Z)^\times \arrow{d}{\beta} \\
     & \C^\times
  \end{tikzcd}
 \end{equation} 

\noindent We recall Theorem 3.4.1 of \cite{pierceschmidt2025siegel}:

\begin{theorem}\label{Fourierexpansiontheorem}
 Let $k\geq4$ be an integer, and let $\eta$ be a primitive Dirichlet character of conductor $N=\prod p^{n_p}$. Then the function 
 \begin{equation}\label{eketadef}
 E_{k,\eta}(Z):=\sum_{b\in(\Z/N\Z)^\times/\{\pm1\}}\eta(b)E_k(Z,K(N^2),\mathcal{C}_0(bN)),
\end{equation} has a Fourier expansion
 \begin{equation}\label{Fourierexpansiontheoremeq1}
  E_{k,\eta}(Z)=\delta_{\eta=1}+\sum_{T\neq0}a(T)\ee^{2\pi\ii\trace(TZ)},
 \end{equation}
 where $T$ runs over non-zero positive semi-definite matrices $\mat{n}{r/2}{r/2}{m}$ with integers $n,r,m$ such that $N^2\mid m$, and where, for $\mathrm{rank}(T)=1$,
 \begin{align}\label{Fourierexpansiontheoremrank1}
  a(T)&=\frac{(-2\pi \ii)^k}{(k-1)!}
  \begin{cases}
    \displaystyle\frac{ \sigma_{k-1}(n)}{\zeta(k)}  & \text{if }N=1,\:T=\mat{n}{0}{0}{0},\:n>0,\\[1ex] 
   \displaystyle \frac{\sigma_{k-1,\eta}(e_{\scriptscriptstyle \hat N})}{L(k,\eta)}\frac{\eta(r_{\scriptscriptstyle{\hat N}})}{\eta(2_{\scriptscriptstyle{\hat N}})}e_{\scriptscriptstyle N}^{k-1}
   &\text{if }m>0,\:r_{\scriptscriptstyle N}=\frac{(2m)_{\scriptscriptstyle N}}{N},\\
   0&\text{otherwise},
  \end{cases}
\end{align}
and for $\mathrm{rank}(T)=2$,
 \begin{align}\label{Fourierexpansiontheoremrank2}
  &a(T)=\frac{(4\pi)^{2k-1}\det(T)^{k-\frac32}}{2(2k-2)!}N^{2-2k} f_{\scriptscriptstyle\hat N}^{3-2k}\eta( f_{\scriptscriptstyle\hat N}^2)\tilde H_{D,k,\eta}( e_{\scriptscriptstyle\hat N}, f_{\scriptscriptstyle\hat N})\nonumber\\
    &\qquad\frac{L(k-1,\chi_D\eta)}{L(k,\eta)L(2k-2,\eta^2)}G(\eta)\Bigg(\prod_{\substack{p<\infty\\p\mid N\\p\nmid r}}\chi_p(r)\Bigg)\Bigg(\prod_{\substack{p<\infty\\p\mid N\\p\mid r}}p^{n_p(2-k)}\chi_p(p^{n_p})K(k,T,\chi_p)\Bigg).
 \end{align}
 Here, $r^2-4nm=Df^2$ with $D$ a fundamental discriminant and $f$ an integer, $e=\gcd(n,r,m)$, and we employ the notations \eqref{Nnotation} and \eqref{divisorsumwithcharacter}.
\end{theorem}

To prove Theorem \ref{maintheorem}, we start with the following proposition.

\begin{proposition}\label{GaussalphasqrtDprop}
     Suppose $r^2-4mn=-Df^2$ for $D<0$ a fundamental discriminant and $f$ an integer.  Let $\alpha$ be the primitive Dirichlet character $\bmod$ $a$ corresponding to $\chi_D\eta$ as explained above.  Then
     \begin{equation}
         G(\alpha)\sqrt{|D|}\in\Q(\eta,\zeta_N,\ii).
     \end{equation}
 \end{proposition}
 
 \begin{proof}
      Let the quadratic character $\chi_D$ correspond to the adelic character $\sigma=\otimes\sigma_p$ so that $\alpha$ corresponds to $\chi\sigma=\otimes\chi_p\sigma_p$.  By Proposition \ref{epsilonfactorgausssumprop}, 
 \begin{align}\label{rationalitygausssumeq1}
     &G(\alpha)\sqrt{|D|}=\sqrt{a|D|}\cdot \alpha(-1)\prod\limits_{p\mid a} \varepsilon\left(\tfrac{1}{2},\chi_p\sigma_p,\psi_p\right)\nonumber\\
     &\quad \overset{\eqref{epsilonfactordef}}{=}\sqrt{a|D|}\cdot\alpha(-1)\left(\prod\limits_{p\mid a}p^{a(\chi_p\sigma_p)/2}(\chi_p\sigma_p)(p^{v_p(a)})\right)\prod\limits_{p\mid a}\, \int\limits_{\Z_p^\times}\chi_p(\mu)^{-1}\sigma_p(\mu)^{-1}\psi_p(p^{-v_p(a)}\mu)\, d\mu.
 \end{align}   We define
 \begin{align}
     A_p &:=\sqrt{p^{v_p(a)}p^{v_p(|D|)}}\, \varepsilon\left(\tfrac{1}{2},\chi_p\sigma_p,\psi_p\right)\nonumber\\
     &=\sqrt{p^{v_p(a)}p^{v_p(|D|)}}\, p^{a(\chi_p\sigma_p)/2}(\chi_p\sigma_p)(p^{v_p(a)})\int\limits_{\Z_p^\times}\chi_p(\mu)^{-1}\sigma_p(\mu)^{-1}\psi_p(p^{-v_p(a)}\mu)\, d\mu.
 \end{align} It suffices to show that for each prime $p\mid a$, $A_p\in \Q(\eta,\zeta_N,\ii)$. Fix $p\mid a$.  Since $\sigma_p$ is quadratic,
 \begin{equation}\label{rationalityeq19}
     \varepsilon\left(\tfrac{1}{2},\sigma_p,\psi_p\right)\overset{\eqref{rationalityeq17}}{\in}\Q(\ii).
 \end{equation}  If $p\nmid N$, $\chi_p$ is unramified and
 \begin{equation}\label{rationalityeq20}
     \varepsilon\left(\tfrac{1}{2},\chi_p\sigma_p,\psi_p\right)\overset{\eqref{rationalityeq18}}{=}\chi_p(p^{-v_p(a)})\varepsilon\left(\tfrac{1}{2},\sigma_p,\psi_p\right)\in\Q(\eta,\ii).
 \end{equation}  We also note that if $p\mid N$ and $\chi_p$ is quadratic, then $\chi_p\sigma_p$ is quadratic and 
 \begin{equation}\label{rationalityeq21}
     \varepsilon\left(\tfrac{1}{2},\chi_p\sigma_p,\psi_p\right)\overset{\eqref{rationalityeq17}}{\in}\Q(\ii).
 \end{equation}
 Suppose first that \underline{$p\nmid N$ and $p\mid |D|$}. Then $v_p(a)=a(\chi_p\sigma_p)=a(\sigma_p)=v_p(D)$.  So
 \begin{equation}
     A_p=p^{v_p(a)}\varepsilon\left(\tfrac{1}{2},\chi_p\sigma_p,\psi_p\right)\overset{\eqref{rationalityeq20}}{\in}\Q(\eta,\ii).
 \end{equation} Next, suppose \underline{$p\mid N$ and $p\nmid |D|$}.  Then $v_p(a)=a(\chi_p\sigma_p)=a(\chi_p)=v_p(N)$ and $v_p(D)=0$.  We have
  \begin{align}
      A_p =p^{v_p(N)}(\chi_p\sigma_p)(p^{v_p(N)})\int\limits_{\Z_p^\times}\chi_p(\mu)^{-1}\sigma_p(\mu)^{-1}\psi_p(p^{-v_p(N)}\mu)\, d\mu\in\Q(\eta,\zeta_N),
  \end{align} since for $\mu\in\Z_p^\times$, $\psi_p$ takes values in $\langle \zeta_{p^{v_p(N)}}\rangle$.

 Suppose that \underline{$p\mid N$, $p\mid |D|$, and $\chi_p^2\neq 1$}. First assume $p$ is odd.  Then $v_p(a)=a(\chi_p)=v_p(N)$ and $v_p(D)=1$.  By \eqref{sqrtoddp}, $A_p\in\Q(\eta,\zeta_N)$.  Now assume $p=2$.  Then $a(\sigma_p)\leq 3<4\leq a(\chi_p)$ so that $v_p(a)=v_p(N)$ and $v_p(D)\in\{2,3\}$.  Then if suffices to show 
  \begin{equation}
      \sqrt{2^{v_2(D)}}\in\Q(\zeta_N,\ii).
  \end{equation}  This is clear if $v_2(D)=2$.  If $v_2(D)=3$, we are done by \eqref{sqrtevenp}.  

  Finally, suppose \underline{$p\mid N$, $p\mid |D|$, and $\chi_p^2= 1$}.  By \eqref{rationalityeq21}, it suffices to show \begin{equation}
      \sqrt{p^{v_p(a)}p^{v_p(D)}}~\in~\Q(\zeta_N).
  \end{equation}  If $v_p(a)+v_p(D)$ is even, this is clear.  Otherwise, it is enough to show $\sqrt{p}\in\Q(\zeta_N)$. For odd $p$, this is true by \eqref{sqrtoddp}.  If $p=2$, 
   $v_2(a)+v_2(D)$ is odd only if $v_2(N)=3$ and then \eqref{sqrtevenp} applies.  This completes the proof.
 \end{proof}

Now we prove the main theorem.

\begin{proof}[Proof of Theorem~\ref{maintheorem}]
    For part $i)$, notice the Gauss sum defined in \eqref{Gausssumeq1} takes values in $\Q(\eta,\zeta_N)$.  If $N=1$, it is known that the Fourier coefficients are rational (see \cite{Siegel1939}).  We will assume throughout that $N>1$.

From Theorem \ref{Fourierexpansiontheorem}, we see that in rank 0, $a(T)\in\{0,1\}\subseteq \Q$.  In rank 1, $a(T)$ is as in \eqref{Fourierexpansiontheoremrank1}.  Suppose $m>0$ and $r_{\scriptscriptstyle N}=\frac{(2m)_{\scriptscriptstyle N}}{N}$.  Then,
\begin{align}\label{rationalityeq6}
    a(T) &\overset{\eqref{specialvalueseq6}}{=}\frac{(-1)^{-1+\frac{k+\epsilon}2}2k\ii^{k+\epsilon}m^k \sigma_{k-1,\eta}(e_{\scriptscriptstyle \hat N})}{G(\eta)B_{k,\bar\eta}}\frac{\eta(r_{\scriptscriptstyle{\hat N}})}{\eta(2_{\scriptscriptstyle{\hat N}})}e_{\scriptscriptstyle N}^{k-1}\in\Q(\eta,\zeta_N),
\end{align} since $\sigma_{k-1,\eta}, B_{k,\bar\eta}\in\Q(\eta)$ and $k+\epsilon$ is always even.

In rank 2, $a(T)$ is as in \eqref{Fourierexpansiontheoremrank2}.  It is easy to see that 
 \begin{equation}\label{rationalityeq8}
     N^{2-2k} f_{\scriptscriptstyle\hat N}^{3-2k}\eta( f_{\scriptscriptstyle\hat N}^2)\tilde H_{D,k,\eta}( e_{\scriptscriptstyle\hat N}, f_{\scriptscriptstyle\hat N})\Bigg(\prod_{\substack{p<\infty\\p\mid N\\p\nmid r}}\chi_p(r)\Bigg)\Bigg(\prod_{\substack{p<\infty\\p\mid N\\p\mid r}}p^{n_p(2-k)}\chi_p(p^{n_p})\Bigg)\in\Q(\eta).
 \end{equation}  From Proposition \ref{rationalityofKprop}, $K(k,T,\chi_p)\in\Q(\eta)$.  So it remains to show
 \begin{equation}\label{rationalityeq9}
     \frac{(4\pi)^{2k-1}\det(T)^{k-\frac32}}{2(2k-2)!}\frac{L(k-1,\chi_D\eta)}{L(k,\eta)L(2k-2,\eta^2)}G(\eta)\in\Q(\eta,\zeta_N,\ii).
 \end{equation}   Since the $L$-functions of a Dirichlet character and its associated primitive Dirichlet character differ only by finitely many Euler factors, it is easy to see that
 \begin{equation}\label{rationalityeq12}
     \frac{L(k-1,\chi_D\eta)}{L(k-1,\alpha)}\in\Q(\eta) \quad \text{ and } \quad \frac{L(2k-2,\eta^2)}{L(2k-2,\beta)}\in\Q(\eta).
 \end{equation}
 To prove \eqref{rationalityeq9}, it now suffices to consider
 \begin{align}\label{rationalityeq14}
     &\frac{(4\pi)^{2k-1}\det(T)^{k-\frac32}}{2(2k-2)!}\frac{L(k-1,\alpha)}{L(k,\eta)L(2k-2,\beta)}G(\eta)\nonumber\\
     &\qquad\qquad\qquad \overset{\eqref{specialvalueseq6}}{=}\pm 2^{2k-1}\det(T)^{k-\frac32}\ii^{(-1)^{k+1}}\frac{N^k b^{2k-2}k}{a^{k-1}}\frac{B_{k-1,\bar\alpha}}{B_{k,\bar\eta}B_{2k-2,\bar\beta}}\frac{G(\alpha)}{G(\beta)}\nonumber\\
     &\qquad\qquad\qquad =\pm 4f^{2k-3}|D|^{k-\frac32}\ii^{(-1)^{k+1}}\frac{N^k b^{2k-2}k}{a^{k-1}}\frac{B_{k-1,\bar\alpha}}{B_{k,\bar\eta}B_{2k-2,\bar\beta}}\frac{G(\alpha)}{G(\beta)}.
 \end{align}  Observing that $G(\beta)\in \Q(\beta,\zeta_b)\subseteq \Q(\eta,\zeta_N)$ and $G(\alpha)\sqrt{|D|}\in \Q(\eta,\zeta_N,\ii)$ by Proposition \ref{GaussalphasqrtDprop}, we have now proven the claim \eqref{rationalityeq9} and hence part $i)$ of the theorem.  Part $ii)$ is easy to see.  
\end{proof}

\bibliography{my}{}

\begin{thebibliography}{1}

\bibitem{Bump1997}
Daniel Bump.
\newblock {\em Automorphic forms and representations}, volume~55 of {\em Cambridge Studies in Advanced Mathematics}.
\newblock Cambridge University Press, Cambridge, 1997.

\bibitem{Neukirch1999}
J\"{u}rgen Neukirch.
\newblock {\em Algebraic number theory}, volume 322 of {\em Grundlehren der Mathematischen Wissenschaften [Fundamental Principles of Mathematical Sciences]}.
\newblock Springer-Verlag, Berlin, 1999.
\newblock Translated from the 1992 German original and with a note by Norbert Schappacher, With a foreword by G. Harder.

\bibitem{Pierce2026}
Erin Pierce.
\newblock On {S}iegel {E}isenstein series with level.
\newblock {\em University of North Texas PhD Thesis}, 2026.

\bibitem{pierceschmidt2025siegel}
Erin Pierce and Ralf Schmidt.
\newblock Siegel {E}isenstein series with paramodular level.
\newblock {\em arXiv preprint arXiv:2509.04395}, 2025.

\bibitem{schmidtsahapitale2021}
Ameya Pitale, Abhishek Saha, and Ralf Schmidt.
\newblock On the standard {$L$}-function for {$\rm{GSp}_{2n}\times\rm{GL}_1$} and algebraicity of symmetric fourth {$L$}-values for {$\rm{GL}_2$}.
\newblock {\em Ann. Math. Qu\'e.}, 45(1):113--159, 2021.

\bibitem{shimura2000arithmeticity}
Goro Shimura.
\newblock {\em Arithmeticity in the theory of automorphic forms}, volume~82.
\newblock American Mathematical Society Providence, RI, 2000.

\bibitem{Siegel1935}
Carl~Ludwig Siegel.
\newblock {\"U}ber die analytische {T}heorie der quadratischen {F}ormen.
\newblock {\em Ann. of Math. (2)}, 36(3):527--606, 1935.

\bibitem{Siegel1939}
Carl~Ludwig Siegel.
\newblock Einf\"uhrung in die {T}heorie der {M}odulfunktionen {$n$}-ten {G}rades.
\newblock {\em Math. Ann.}, 116:617--657, 1939.

\bibitem{TateThesis}
John Tate.
\newblock Fourier analysis in number fields, and {H}ecke's zeta-functions.
\newblock {\em Algebraic Number Theory (Proc. Instructional Conf., Brighton, 1965)}, pages 305--347, 1950.

\end{thebibliography}
\bibliographystyle{plain}

\end{document}